\documentclass[11pt,a4paper,draft]{article}
\usepackage{anysize}
\usepackage{paralist}
\usepackage{amssymb}
\usepackage{amsthm}
\usepackage{amscd}
\usepackage{amsmath}
\newcommand{\Z}{{\mathbb Z}}
\newcommand{\N}{{\mathbb N}}
 \newcommand\isom{\cong}
 \newcommand{\tens}{\otimes}
 \newcommand{\deq}{:=}
\newcommand{\cmps}{\circ}
\newcommand\set[1]{\left\{#1\right\}}
 \newcommand\simp[1]{\langle#1\rangle}
\newcommand\xto{\xrightarrow}
 \DeclareMathOperator\Bild{im}
 \DeclareMathOperator\Kern{ker}
 \newcommand\im{\Bild}
 \renewcommand\ker{\Kern}
 \DeclareMathOperator\sign{sign}
 \newcommand\join{\ast}
\usepackage{hyperref}
\usepackage{graphicx}
\newcommand{\wasauchimmer}{\rotatebox[origin=c]{270}{\ensuremath{\in}}}
\newcommand{\dd}{\partial}

\newcommand\alt{{\mathrm{alt}}}
\newcommand\sm{{\setminus}}

\newtheorem{theorem}{Theorem}[section]
\newtheorem{corollary}[theorem]{Corollary}
\newtheorem{lemma}[theorem]{Lemma}
\newtheorem{proposition}[theorem]{Proposition}
\theoremstyle{definition}
\newtheorem{definition}[theorem]{Definition}
\newtheorem{example}[theorem]{Example}

\newtheorem{remark}[theorem]{Remark}
\newtheorem{remark*}{Remark}

\begin{document}
\title{Combinatorial Stokes formulas\\
       via minimal resolutions}
\author{Bernhard Hanke\\Inst.\ Math.\\LMU M\"unchen\\
         \texttt{hanke@mathematik.uni-muenchen.de}
 \and    Raman Sanyal\\Inst.\ Math., MA 6-2\\TU Berlin\\
         \texttt{sanyal@math.tu-berlin.de}
 \and    Carsten Schultz\\Inst.\ Math., MA 6-2\\TU Berlin\\
         \texttt{cschultz@math.tu-berlin.de}
 \and    G\"unter M. Ziegler\\Inst.\ Math., MA 6-2\\TU Berlin\\
         \texttt{ziegler@math.tu-berlin.de}
}
\date{September 29, 2007}
\maketitle

\begin{abstract}
  We describe an explicit chain map from the standard resolution to the
  minimal resolution for the finite cyclic group $\Z_k$ of order~$k$.  We
  then demonstrate how such a chain map induces a ``$\Z_k$-combinatorial
  Stokes theorem'', which in turn implies ``Dold's theorem'' that there is
  no equivariant map from an $n$-connected to an $n$-dimensional free
  $\Z_k$-complex.

  Thus we build a combinatorial access road to problems in combinatorics and
  discrete geometry that have previously been treated with methods from
  equivariant topology.
  The special case $k=2$ for this is classical; it involves
  Tucker's (1949) combinatorial lemma which implies the Borsuk--Ulam
  theorem, its proof via chain complexes by Lefschetz (1949),
  the combinatorial Stokes formula of Fan (1967), and
  Meunier's work (2006).
\end{abstract}

\section{Introduction}

The Borsuk--Ulam theorem \cite{Borsuk} about $\Z_2$-equivariant maps
between spheres,
and its extension to $\Z_k$-actions formulated by Dold \cite{Dold},
have many interesting applications in combinatorics and geometry ---
see Matou\v{s}ek \cite{Matousek:BU}.
Since these are topological theorems with purely combinatorial
consequences, there is great interest in combinatorial approaches
to the area.

\subsection{The classical case, \boldmath$k=2$}

For the case $k=2$ such a path-way is well-established:
In 1945, Tucker \cite{Tucker} presented a combinatorial lemma that
implies the Borsuk--Ulam theorem: A centrally symmetric triangulation of~$S^n$
that refines the hyperoctahedral triangulation of the $n$-sphere cannot get an
antipodal labelling from the set $\{\pm1,\dots,\pm n\}$ such that no edge gets
vertex labels $+i,-i$.  In 1952, Fan~\cite{fan52} extended this lemma: If
the labels are taken from the set $\{\pm1,\dots,\pm m\}$, then the number of
facets of the triangulation of~$S^n$ that get an ``alternating labelling'' by
$+j_0,-j_1,\ldots,(-1)^nj_n$ with $1\le j_0<j_1<\dots<j_n\le m$ is odd and
hence non-zero. In particular, $m$ must be larger than $n$ for such a
labelling to exist.

In 1952, Fan \cite{fan67} presented a rainbow coloring theorem for general
pseudomanifolds (interpreted as a ``combinatorial Stokes theorem'' by Meunier
\cite{meunier06:_spern}), which says that for any orientable $n$-dimensional
pseudomanifold with boundary, equipped with a coloring by $\{\pm1,\dots,\pm
m\}$ without antipodal edges, the number of rainbow-colored $n$-simplices with
positive smallest label equals the number of rainbow-colored $(n-1)$-simplices
in the boundary (counted with appropriate signs, depending on dimension and
orientations).  The resulting formula is easy to prove since by linearity it
can be reduced to the case of a pseudomanifold that consists of a single
$n$-simplex. However, a treatment in terms of chain complexes yields a
simple, systematic
proof that also motivates the formula in question; this was first done in
Lefschetz' 1949 treatment \cite[Sect.~IV\S7, pp.~134--140]{Lefschetz-intro} of
Tucker's lemma, and then for Fan's lemma by Meunier \cite{meunier06:_spern}.
This also leads to simple, transparent, \emph{combinatorial} proofs for the
Kneser conjecture (see Matou\v{s}ek \cite{Mat1}, Ziegler \cite{Z77}) and for
its strengthening by Schrijver \cite{Schrijver-kneser} (see
Meunier~\cite{meunier06:_spern}).%

As amply demonstrated in Matou\v{s}ek \cite{Matousek:BU}, a variety of
combinatorial hypergraph coloring problems as well as various geometric
multiple-incidence problems were first proved by a result known as Dold's
theorem \cite{Dold}, which says that there is no equivariant map from an
$n$-connected free $\Z_k$-complex to an $n$-dimensional such complex. (For
$k=2$ this is equivalent to the Borsuk--Ulam theorem). In view of the purely
combinatorial hypergraph coloring results proved with this tool (see Alon,
Frankl \& Lov\'asz \cite{AlonFranklLovasz}, Matou\v{s}ek \cite{Mat2}, Ziegler
\cite{Z77}, etc.), one is led to ask for an analogous combinatorial treatment
of Dold's theorem, for a ``$\Z_k$-Tucker lemma'', etc.
Steps in this direction were taken by Ziegler \cite{Z77} and in particular by
Meunier \cite{meunier:_z_fan05}, who obtained a semi-explicit
combinatorial Stokes formula for the case when $k$ is odd.

\subsection{The \boldmath$\Z_k$-combinatorial Stokes theorem}%
\label{ssec:plan}

The main objective of this paper is not only to derive a
``$\Z_k$-combinatorial Stokes formula'', Theorem~\ref{thm:stokes},
that is valid for all $k\ge2$, but also to
explain where such a result comes from, and why it has the form it has.
This question arises even in the classical case of $k=2$:
\emph{Why} should we look for, and count, simplices with
alternating labels, with signs that depend on
parity of dimension and on orientation?

A hint for this is given by Meunier's treatment
in~\cite{meunier06:_spern} of Fan's
combinatorial Stokes theorem, via chain complexes:
The chain complex that plays a prominent role in his proof is the
minimal free resolution (in the group homology sense) of the group
$\Z_2$, and Meunier's proof in essence builds on a
$\Z_2$-equivariant chain map from the chain complex of the universal label space to the minimal resolution.

Our combinatorial Stokes formula
concerns simplical complexes $X$ whose vertices get labels in the set
$\Z_k\times\N$, where we interpret the elements of $\N$ as ``colors'',
while the elements of $\Z_k$  play the role of
``signs''. The main requirement is that adjacent vertices of~$X$
may not have the same color and different signs.
Such an \emph{admissible} labelling $\ell:V(X)\rightarrow\Z_k\times\N$
amounts to a simplicial map from $X$ to a
``universal label space''  $(\Z_k)^{*\N}$ and this establishes a chain map
$\ell_\#:C_\bullet(X)\rightarrow C_\bullet((\Z_k)^{*\N})$ of simplicial
chain complexes (with coefficient in some commutative ring $R$).

The label space is equipped with a canonical free simplicial $\Z_k$-action, corresponding
to the natural symmetry of admissible labellings given  by cyclically permuting the signs in $\Z_k$.
Thus there is a $\Z_k$-equivariant chain map
$C_\bullet((\Z_k)^{*\N})\rightarrow M_\bullet$ to the
\emph{minimal resolution} of the ring $R$ over the
group ring $R[\Z_k]$ which commutes
with the canonical augmentations on both complexes,
unique up to $\Lambda$-linear chain homotopy.
This statement relies on the fact that $M_\bullet$ is
a \emph{free resolution} of $R$ over $R[\Z_k]$.
The chain complex $M_\bullet$ consists of free modules of rank one over $R[\Z_k]$ in every degree,  
hence only label patterns of a very specific form survive to the minimal resolution.

The combinatorial Stokes formula results from an explicit
description of the chain map to the minimal resolution (and
in particular of the surviving label patterns) combined with the simple fact that this
chain map commutes with boundary operators.

The following diagram of chain complexes and chain maps illustrates
the homological-algebraic content of this mechanism.

\[
\begin{array}{ccccc@{\qquad\quad}l}
                 &  x       &   \stackrel{\partial_i}\longmapsto &
                 \partial x  \\
 & \wasauchimmer && \wasauchimmer
\\
\rightarrow  &     C_i(X)   & \stackrel{\partial_i}\longrightarrow                 &
                   C_{i-1}(X)   & \rightarrow & \textrm{\small simplicial chain complex} \\[2mm]
                 & \ell_{\#}  \downarrow&& \ell_{\#}  \downarrow &
& \quad  \textrm{\scriptsize (labelling)}\\[2mm]
\rightarrow  & C_i((\Z_k)^{*\N})& \stackrel{\partial_i}\longrightarrow &
                   C_{i-1}((\Z_k)^{*\N})& \rightarrow &
                   \textrm{\small chain complex of label space}\\[2mm]
                 &  h_i  \downarrow&& h_{i-1}  \downarrow &
& \quad   \textrm{\scriptsize (map to standard resolution)}  \\[2mm]
\rightarrow  & S_i &\stackrel{\partial_i}\longrightarrow & S_{i-1} & \rightarrow &
                   \textrm{\small standard resolution}\\[2mm]
                 &  f_i  \downarrow&& f_{i-1}  \downarrow &
& \quad \textrm{\scriptsize (map to minimal resolution, Prop.~\ref{prop:f-chain-map}) }   \\[2mm]
\rightarrow  & M_i=R[\Z_k] &\stackrel{\partial_i}\longrightarrow & M_{i-1}= R[\Z_k]  & \rightarrow &
                     \textrm{\small minimal resolution}\\[2mm]
                 &&& u\downarrow &
&\quad \textrm{\scriptsize  (evaluate at the neutral element)}\\[2mm]
                 &&& R
\end{array}
\]

The composite chain map $h_\bullet^\ell:= h_\bullet \circ \ell_\#$ (see Section \ref{sec:labelled_complexes})
sends simplices to ``patterns'' of label sequences (counted with multiplicities and
according to orientation).
Thus, $h^\ell_i (x)$ is the formal sum of the
patterns that arises from an $i$-chain $x\in C_i(X)$, while
$h^\ell_{i-1}(\partial_i x)$
is the corresponding sum of patterns on the
$(i-1)$-simplices in the boundary of~$x$.

Given a pattern for $i$-simplices,
the map $f_i$ to $M_i$ followed by the
boundary map of the minimal resolution and then by the evaluation map $u$
(which maps an element of the group ring to the coefficient of the
neutral element) tells us how to count $i$-patterns.
Similarly, we count $(i-1)$-patterns according to $u\circ f_{i-1}$.

In this notation, the combinatorial Stokes formula
simply reads
\[
    u \big( ( \partial \cmps f \cmps h^\ell)  (x) \big)  =   u \big(  ( f \cmps h^\ell)   (\partial x)  \big) \in \ R
\]
for $x\in C_i(x)$. Our Theorem~\ref{thm:stokes}
combines this fact  with the explicit description of the chain map $f_\bullet $
presented in Section \ref{sec:chain_map}.

The formula obtained in this way depends on some choices.
Indeed, the map $\ell_\# $ is determined by the given labelling on $X$
and there is a canonical choice for $h_\bullet$. Furthermore,
replacing $u$ by the evaluation at another group element in $\Z_k$
induces a Stokes formula which is given by shifting the signs
involved in the old one cyclically by the inverse of this element.
However, the map $f_\bullet$ is determined only up to chain
homotopy and different choices lead to
different Stokes formulas, in general.
It is easy to see (cf. Lemma \ref{lem:uniqueness}) that
the chain map from the standard to the minimal resolution is uniquely determined
upon choosing  $R$-linear complements of the kernels of the boundary operator
in each degree of the minimal resolution.
We will propose a particular choice,  uniform for all $k$
(see the remarks following Lemma \ref{lem:uniqueness}),
and analyze the corresponding label patterns surviving to the minimal resolution
in terms of {\em strongly alternating labellings} (see Definition \ref{def:stralt}). This notion and
the resulting Stokes formula
restrict to the notion of alternating labellings and to
the classical Fan formula if $k=2$.

The boundary operator $\partial_i$ in the minimal resolution depends on the parity of $i$.
Consequently, as in the classical case $k=2$,
we actually get two combinatorial Stokes formulas depending on whether the
dimension of the given simplicial chain on $X$ is even or odd.

\subsection{Plan}

In Section~\ref{sec:basics} we review the
combinatorial Stokes formula and the Tucker lemma in the classical
case when $k=2$. In this case $X$ is required to be a
$d$-pseudomanifold, and $x=o_d\in C_d(X)$ is an orientation chain for
it. However, a key example for our discussions for arbitrary $k$ is the
universal label space $(\Z_k)^{*m}$,
and this is a pseudomanifold for $k=2$ (at least for finite $m$), but not for
$k>2$. Thus we admit for greater generality below.

Section~\ref{sec:chain_map} is the technical heart of 
our paper: We explicitly construct the
chain maps that lead to the combinatorial Stokes formula
in Section~\ref{sec:ZpStokes} and we give a combinatorial interpretation of 
the relevant label patterns in terms of strongly alternating elements. 
We remark that this construction is much more difficult for $k >2$ than 
in the classical case $k=2$. 

From this, in Section~\ref{sec:Tuckerlemmas}, we derive
``$\Z_k$-Tucker lemmas''.
What should such a result achieve, if we follow the model for $k=2$?
It should refer to a labeled simplicial complex $X$
\emph{with a free $\Z_k$ action}, and predict the
existence of simplices with a specified type of label pattern.
Topologically, it should imply
that for some $d$-connected free $\Z_k$-space
with arbitrarily fine triangulation
(for $k=2$: antipodal triangulations of the $d+1$-sphere)
there is no equivariant map to a specific $d$-dimensional free
$\Z_k$-space which serves as a ``label space''.
The Tucker lemmas should be derived from the combinatorial Stokes theorem
by induction on the dimension, once we can identify
suitable chains ({\em generalized spheres}, cf. Definition \ref{def:gensphere}) in the complex $X$.
In Section~\ref{sec:Tuckerlemmas}, we derive a
generalized $\Z_k$-Tucker lemma, Theorem~\ref{thm:count-transfer},
which in the case $k=2$ specializes to Fan's and Tucker's lemma,
and which also yields the ``$\Z_k$-Tucker--Fan
lemma'' of Meunier \cite[Thm.~2]{meunier06:_spern}
as an example, without Meunier's
restriction to the case of odd~$k$.
We also derive a (homological) version of Dold's theorem from this set-up.

Finally, in Section~\ref{sec:rainbow} we determine the homotopy type
of the target space $(\Z_k)^{*\N}_{\alt \leq d}$ that appears implicitly in
Meunier's and explicitly in our version of the $\Z_k$-Tucker lemma. In
the special case $k=2$ this yields the natural target space for
rainbow colorings --- which appears in Fan's classical work
\cite{fan82} and its current extensions by Tardos \& Simonyi
\cite{TardosSimonyi3,TardosSimonyi1}.

\begin{small}
\paragraph{Acknowledgements.}

We are grateful to Fr\'ed\'eric Meunier for his unpublished paper
\cite{meunier:_z_fan05}, which was an important stepping-stone for our
work. We also thank Mark de Longueville for important discussions.

This work grew on the hospitality of KTH Stockholm and the
Mittag-Leffler Institute on occasion of the
``Topological Combinatorics'' workshop in May 2005, and the
MSRI special semester ``Computational Applications of Algebraic
Topology'' in the fall of 2007. We gratefully acknowledge support
by a DFG Leibniz grant (GMZ, BH) and from the
DFG Research Training Group ``Methods for Discrete Structures'' in
Berlin (CS, RS). BH thanks TU Berlin for its hospitality.

\end{small}

\section{Fan and Tucker revisited}\label{sec:basics}

A \emph{$d$-pseudomanifold} is a finite, pure $d$-dimensional, simplicial
complex $X$ such that any $(d-1)$-face (ridge) is contained in at
most two $d$-faces (facets) of the complex.  The ridges that lie in
exactly one facet generate the \emph{boundary} $\partial X$, which is thus
a pure $(d-1)$-dimensional simplicial complex (or empty).  The vertex set of a
complex $X$ will be denoted by $V(X)$, the edge set by $E(X)$.
A $d$-pseudomanifold is \emph{orientable} if the facets can be oriented
consistently so that they induce opposite orientations on the interior ridges,
that is, if there is an \emph{orientation $d$-chain} $o_d$ in the chain group
$C_d(X;\Z)$ such that the boundary $\partial o_d$ is supported only on
the boundary complex $\partial X$.

We refer to Munkres \cite{Munkres:AT}
for basics about chain complexes, chain maps, and orientability.

\begin{definition}\label{def:admissible}
    An \emph{admissible vertex labelling} of a pure $d$-dimensional simplicial
    complex $X$ is a map
\[
    \ell:V(X)\longrightarrow\Z\sm\{0\}
\]
    such that no two adjacent vertices obtain opposite labels, that is, such that
    $\ell(v)\neq-\ell(w)$ for $\{v,w\}\in E(X)$.

    Under such a labelling, a \emph{$+$alternating facet} is one that
    obtains labels $+j_0,-j_1,+j_2,\dots,(-1)^dj_d$ with
    $0<j_0<j_2<\dots<j_d$ (that is, all labels have different absolute
    values, and if we order them by absolute value, then the signs alternate, starting
    with a positive sign).  Similarly, a \emph{$-$alternating facet} 
    obtains labels $-j_0,+j_1,-j_2,\dots,(-1)^{d+1} j_d$ with
    $0<j_0<j_2<\dots<j_d$.
\end{definition}

The main result of Fan's 1967 paper \cite{fan67} was that for any admissible
vertex labelling on an oriented
$d$-pseudomanifold, $(-1)^d$ times the number of \emph{$+$alternating facets}
(counted according to orientation and with an additional
minus sign if $d$ is odd) plus the number of
\emph{$-$alternating facets} (counted according to orientation) yields the number of
\emph{$+$alternating} facets in the boundary complex.  Here
``counted according to orientaton'' means that a facet is counted as $-1$ if
the ordering of the vertices according to the label ordering
$j_0,j_1,j_2,\dots,j_d$ yields a negative orientation of the facet (and
similarly for $-$alternating facets).
If the $d$-pseudomanifold is not orientable, then all of this is still true
modulo~$2$.

With or without explicit notation for this (Fan writes
``$\alpha(+j_0,-j_1,+j_2,\dots,(-1)^dj_d)$'' for the number of $d$-simplices
with the given set of labels, counted according to orientation), the precise
count is a bit tricky to digest.  However, it clearly
relates a sum over a pseudomanifold to a sum over the boundary.
This explains why Meunier \cite{meunier06:_spern} calls this a discrete
``Stokes theorem''.

From Fan's lemma, it is easy to derive the Tucker lemma, by induction
on the dimension, using the decomposition of $\Sigma^d$
into upper and lower hemisphere.

\begin{proposition}[Tucker lemma \cite{Tucker}, Lefschetz
{\cite[Sect.~IV\S7]{Lefschetz-intro}}, 
Fan \cite{fan52}]\label{tucker_lemma}
\label{prop:tucker_lemma}
    Let $\Sigma^d$ be a centrally symmetric triangulation of the
    $d$-sphere $S^d$ that refines the hyperoctahedral triangulation. Then
    there is no admissible vertex labelling $\ell: V(\Sigma^d) \rightarrow
    \{\pm 1,\dots,\pm d\}$ that is antipodal, i.e.\ $\ell(-v)=-\ell(v)$
    for all vertices~$v$.

    Indeed, for any antipodal vertex labelling $\ell: V(\Sigma^d)
    \rightarrow \{\pm 1,\dots,\pm m\}$, the number of $+$alter\-nating
    facets (with labels $+j_0,-j_1,+j_2,\dots,(-1)^dj_d$, where $0<j_0<\dots<j_d$)
    is odd and hence nonzero.
\end{proposition}

To match this with the following, and to pave the way for the
transition to a more algebraic treatment, we first re-interpret the
set of labels as
\[
   \Z\setminus\{0\}\ =\ \Z_2\times \N,
\]
where $\N$ are the (non-zero) natural numbers, and
$\Z_2\equiv\{1,-1\}$ (which will later be identified with the multiplicative group of order $2$).

Thus any admissible labelling induces a simplicial map
\[
    \ell: X \ \ \longrightarrow\ \ (\Z_2)^{*\N}.
\]
Here $\Z_2=\{1,-1\}$ is seen as a discrete two element set,
$(\Z_2)^{*m}$ is a simplicial sphere of
dimension $m-1$ (which may be identified with the boundary complex of the
$m$-dimensional cross polytope), and thus the target space
\[
    (\Z_2)^{*\N}\ \ =\ \ \bigcup_{m \ge 1}(\Z_2)^{*m}
\]
is the infinite-dimensional sphere.
The simplicial map $\ell$ induces a map of simplicial chain complexes
\[
   \ell_\#: C_\bullet(X) \rightarrow C_\bullet((\Z_2)^{*\N})
\]
with coefficients in some chosen ring~$R$ (when talking about
orientation classes, this is usually specified to be $\Z$ if the pseudomanifold is orientable, and
$\Z/2$ otherwise).

Here the natural symmetry of admissible label patterns, given by reversing
the signs, comes into play. This amounts to the usual free simplicial $\Z_2$-action
on $(\Z_2)^{*\N}$ and the induced action
on its simplicial chain complex. We now re-interpret this: Taking into account that $(\Z_2)^{*\N}$ 
is a contractible space, the chain complex $C_\bullet((\Z_2)^{*\N})$ is a
free resolution of~$R$ over the group ring $R[\Z_2]$ (see Section~\ref{sec:chain_map}).
It is, however, a huge free resolution, of infinite rank, in each
dimension: The standard basis for $C_i((\Z_2)^{*\N})$
consists of all infinite sequences of type
$({*},{+},{-},{*},{-},{*},\dots\,)$ with exactly $i+1$ non-$*$
elements. By \cite[Lemma~7.4]{brown82:_cohom_group}, there is up to homotopy
a unique chain map  to the \emph{minimal resolution}  which
induces the identity of zero dimensional homology groups (which can be canonically
identified with $R$). For $R = \Z$ the minimal resolution is given by
\[
\cdots
     \stackrel{\left(\begin{smallmatrix}+1&+1\\+1&+1\end{smallmatrix}\right)}\longrightarrow\quad
\Z^2
\quad\stackrel{\left(\begin{smallmatrix}-1&+1\\+1&-1\end{smallmatrix}\right)}\longrightarrow\quad
\Z^2
\quad\stackrel{\left(\begin{smallmatrix}+1&+1\\+1&+1\end{smallmatrix}\right)}\longrightarrow\quad
\Z^2
\quad\stackrel{\left(\begin{smallmatrix}-1&+1\\+1&-1\end{smallmatrix}\right)}\longrightarrow\quad
\Z^2  \quad \longrightarrow \quad 0
\]
with the rightmost $\Z^2$ sitting in degree $0$.
The identification of its zeroth dimensional homology
with $\Z$ is induced by
the map ({\em augmentation}) $\Z^2 \to \Z$ represented by the matrix
$\left( \begin{matrix} +1 & +1 \end{matrix} \right)$.

Each such chain map to the minimal resolution can be factored (up to homotopy) through
the canonical map from $C_\bullet((\Z_2)^{*\N})$ to the so-called {\em standard
resolution} \cite[Sect.~I.5]{brown82:_cohom_group}, by simply deleting the $*$s,
and further through the canonical map from the standard resolution
to the so-called {\em normalized standard resolution}, by throwing away those label
patterns that contain two $+$signs or two $-$signs at consecutive places.

For $k=2$ (but not for larger $k$), the normalized standard resolution is isomorphic
to the minimal resolution. One possible chain map is given by
mapping the alternating sequences
$(+1,-1,+1,\dots)\in(\Z_2)^m$ and $(-1,+1,-1,\dots)\in(\Z_2)^m$
into the first and second copy of $\Z$ in $\Z^2$, respectively.

In view of the later generalization to $\Z_k$, we write $\Z_2=\{e,g\}$ with generator $g$, take $R:=\Z$,
identify $M_i=\Z^2$ with the group ring $\Z[\Z_2]=\Z \cdot e \oplus \Z \cdot g$ for $i\ge0$  
and identify the boundary maps $\partial_i: M_i \to M_{i-1}$ in the minimal resolution with
the multiplication with $\tau:=g-e$ for odd~$i$, and with
the multiplication with $\sigma:=e+g$ for even~$i > 0$.
The augmentation map $M_0 \rightarrow \Z$ is defined as
$\alpha e +  \beta g \mapsto \alpha+\beta$.
We finally define the \emph{evaluation}  at $e \in \Z_2$ by
\[
   u : \Z[\Z_2] \rightarrow \Z  \, , ~~  \alpha e + \beta g \mapsto \alpha \, . 
\]
In summary we get the $\Z_2$-Stokes formula
by interpreting the labelling as a simplicial map, then constructing
the chain map from the chain complex of the color sphere to the
minimal resolution, and then evaluating by~$u$.

It it easily checked that this Stokes formula is identical to
the Fan theorem described after Definition~\ref{def:admissible}.

Replacing $u$ by the evaluation at $g$ yields a second Stokes formula obtained from the previous one
by reversing all signs.

However, there are many other isomorphisms from  the
normalized standard resolution to the minimal resolution. These are in one-to-one
correspondence with $\Z$-linear complements (viewed as graded modules)
of the boundary operator in the minimal resolution.
Consequently, there is no ``canonical'' discrete
Stokes formula, even not in the classical case $k=2$.


\section{Resolutions and a chain map}
\label{sec:chain_map}

Let $k\ge2$.  We denote the cyclic group with $k$~elements by $\Z_k$ and write
it multiplicatively as $\Z_k=\set{e,g,\dots,g^{k-1}}$, where $g$ is a
generator of $\Z_k$.  We work over a commutative ring $R$ with~$1$.
We set $\Lambda=R[\Z_k]$, the group ring of~$\Z_k$ over~$R$.

As usual we consider $R$ as a $\Lambda$-module with $g$ acting trivially.
Questions about $\Z_k$-equivariant maps can often be related to the homology
of the group~$\Z_k$, which is by definition the homology of a chain complex
obtained from a free resolution of~$R$. A {\em free resolution} of~$R$ is an
acyclic chain complex of free $\Lambda$-modules
that is augmented with the (non-free) $\Lambda$-module $R$ in dimension $-1$:
\[
  \cdots\ \  \longrightarrow F_3
    \ \ \stackrel{\dd_3}{\longrightarrow}\ \ F_2
    \ \ \stackrel{\dd_2}{\longrightarrow}\ \  F_1
    \ \ \stackrel{\dd_1}{\longrightarrow}\ \  F_0
    \ \ \stackrel{\dd_0}{\longrightarrow}\ \  R \ \longrightarrow\ 0,
\]
or, equivalently, a free chain complex
\[
F_\bullet : \qquad
 \cdots\ \  \longrightarrow F_3
    \ \ \stackrel{\dd_3}{\longrightarrow}\ \ F_2
    \ \ \stackrel{\dd_2}{\longrightarrow}\ \  F_1
    \ \ \stackrel{\dd_1}{\longrightarrow}\ \  F_0\ \longrightarrow\ 0
\]
such that $H_i(F)=0$ for $i>0$ together with a $\Lambda$-linear
isomorphism ({\em augmentation}) $H_0(F)\xto\isom R$.  In the following we use the latter convention.

For our approach, it is important to describe such resolutions
explicitly.

\begin{definition}[Standard resolution] \label{def:standard}
The \emph{standard resolution} of $R$ is given by
\[
 S_\bullet : \qquad
 \cdots\ \    \longrightarrow S_3
    \ \ \stackrel{\dd_3}{\longrightarrow}\ \ S_2
    \ \ \stackrel{\dd_2}{\longrightarrow}\ \  S_1
    \ \ \stackrel{\dd_1}{\longrightarrow}\ \  S_0 \ \longrightarrow\ 0
\]
with modules
\[
    S_r   \ \ \deq\ \ \underbrace{\Lambda\tens_R\cdots\tens_R\Lambda}_{r+1}
\]
and boundary maps
\[
    \dd_r (h_0\tens\cdots\tens h_r) \ \ \deq\ \
    \sum_{i=0}^r (-1)^i\, h_0\tens\cdots\tens\widehat{h_i}\tens\cdots\tens h_r.
\]
with the (usual) convention that $\widehat{h_i}$ denotes omission from
the tensor product.  The boundary maps are defined on the basis
elements $h_0\tens\cdots\tens h_r$ with $h_1,h_2,\dots,h_r \in \Z_k$
and extended to $R$-linear maps.

The diagonal action $g \cdot (h_0 \tens h_1 \tens \cdots \tens h_r) \deq gh_0
\tens gh_2 \dots \tens gh_r$ turns the modules $S_r$ into
$\Lambda$-modules.
It is easily seen that the boundary maps $\dd_r$ are $\Lambda$-linear.
\end{definition}

\begin{definition}[Bar resolution]
A choice of a special basis of the $S_r$ as $\Lambda$-modules gives rise to the
so called \emph{bar resolution}. This particular basis is given by
\[
    [h_1|h_2|\cdots|h_r]\ \ \deq \ \
        e\tens h_1\tens h_1h_2\tens\cdots\tens h_1h_2\cdots h_r
\]
with $h_1, h_2, \dots, h_r \in \Z_k$.  We allow for $r=0$,
i.e.~$[\,]=e\in S_0$.
\end{definition}

This is clearly a basis of $S_r$ as a
$\Lambda$-module and, for example, the elements of the standard $R$-basis are
rewritten as
\[
    h_0 \tens \cdots \tens h_r\ \ =\ \
    h_0[h_0^{-1}h_1|h_1^{-1}h_2|\dots|h_{r-1}^{-1}h_r].
\]

In this basis, the boundary is given by
\begin{eqnarray*}
    \dd_r[h_1|h_2|\cdots|h_r]
        &=&  h_1[h_2|\cdots|h_r] \\
        &&+\ \  \sum_{i=1}^{r-1} (-1)^i
                [h_1|\cdots|h_{i-1}|h_ih_{i+1}|h_{i+2}|\cdots|h_r] \\
        &&+\ \  (-1)^r[h_1|h_2|\cdots|h_{r-1}].
\end{eqnarray*}

\begin{definition}[Minimal resolution]
The \emph{minimal resolution}  is given by
\[  M_\bullet : \qquad
 \cdots \ \  \longrightarrow M_3
    \ \ \stackrel{\dd_3 = m_\tau}{\longrightarrow}\ \ M_2
    \ \ \stackrel{\dd_2 = m_\sigma}{\longrightarrow}\ \  M_1
    \ \ \stackrel{\dd_1 = m_\tau}{\longrightarrow}\ \  M_0\ \longrightarrow\ 0
\]
with $M_i := \Lambda$ for all $i \ge 0$. The boundary maps are defined by
\[
\dd_r\ \ :=\ \
\begin{cases}
    m_\sigma, & \text{if $r$ is even,}\\
    m_\tau,   & \text{if $r$ is odd,}
\end{cases}
\]
where $m_x$ denotes multiplication by $x \in \Lambda$ and
\[\tau = g - e \, , ~~ \sigma = e + g + \cdots + g^{k-1}.\]
\end{definition}

More generally, we define elements
\[
  \tau_r := g^r - e \, , ~~  \sigma_r := e + g + \ldots + g^{r-1} 
\]
for $0 \leq r \leq k$. In particular, $\sigma_0 = 0$, $\sigma_k = \sigma$, $\tau_1 = \tau$, and $\tau_0 = \tau_k = 0$.
The sets
\[
    \Sigma := \{ \sigma_1, \sigma_2, \dots, \sigma_k \} \, , ~~
      T    := \{ e, \tau_1, \tau_2, \dots, \tau_{k-1} \} 
\]
are both bases of
$\Lambda$ as an $R$-module and we have
the identities
\[
    \tau\sigma_i  = \tau_i \, , ~~ \sigma \tau_i = 0
\]
for $1 \leq i \leq k$.
It will therefore be useful to represent $M_i$ in the basis $T$ for even $i$ and in the basis $\Sigma$
for odd $i$.  This choice is justified by the identities
\[
\begin{array}{r@{=}l@{=}l}
    \ker m_\sigma \ \ & \ \  R\tau_1 \oplus R\tau_2 \oplus \cdots \oplus
    R\tau_{k-1} \ \ &\ \  \im m_\tau \\
    \im m_\sigma \ \ &\ \ R \sigma_k \ \ &\ \ \ker m_\tau
\end{array}
\]
which prove that $M_\bullet$ is exact in positive dimensions and, indeed,
a free resolution of $R$ with an augmentation $M_0 \to R$ defined by
\[
      \sum_{i=0}^{k-1} \alpha_i g^i\ \ \mapsto\ \ \sum_{i=0}^{k-1} \alpha_i \, .
\]
Because $S_\bullet$ and $M_{\bullet}$ are free resolutions,
there is a  $\Lambda$-linear chain map $S_\bullet \to M_\bullet$
which is augmentation preserving (and indeed identifies $S_0$ and $M_0$ canonically).  
This chain map is unique up to chain homotopy, 
see \cite[Lemma~7.4]{brown82:_cohom_group}.
The following lemma, which is proved by an easy inductive argument,
shows how we can achieve uniqueness in this situation.

\begin{lemma} \label{lem:uniqueness}
Let $K_r \subset M_r$, $r \geq 0$,  be a collection of $R$-submodules
so that the module $K_{r}$ is an $R$-complement of $\ker \partial_r$ for all
$r \geq 0$ (here,
$\partial_0 : M_0 \to R$ is the augmentation). 
Then there is a unique augmentation preserving
$\Lambda$-linear chain map
\[
    S_\bullet \to M_\bullet
\]
which sends the basis elements $[h_1|h_2|\cdots|h_r]$ from the bar resolution into $K_r$.
\end{lemma}

The $R$-bases  $T$ and $\Sigma$ of $\Lambda$ introduced above motivate a
feasible choice for such a complementary graded submodule $K_\bullet \subset M_\bullet$: For $s \geq 0$ we
set
\begin{eqnarray*}
     K_{2s} & := & R e \, , \\
     K_{2s+1} & := & R \sigma_1 \oplus \ldots \oplus R \sigma_{k-1} \, .
\end{eqnarray*}
Note that for $k=2$, this specializes to $K_i := Re$ for all $i \geq 0$.

Our aim is to give an explicit description of the resulting chain  map $S_\bullet \to M_\bullet$.
It relies on the following notion.

\begin{definition}[Strongly alternating elements] \label{def:stralt}
    Let $h_0, h_1, \dots, h_{2s}\in \Z_k$. We call
    the element $h_0 \tens \cdots \tens h_{2s}$ of~$S_{2s}$ \emph{strongly alternating} if
    its bar representative
    \[
        h_0 \tens \cdots \tens h_{2s}\ \ =\ \ g^{a_0} [g^{a_1}|\cdots|g^{a_{2s}}],
    \]
    with $0 \le a_i < k$ for all $i = 0, \dots, 2s$, satisfies
    \[
    a_{2i+1} + a_{2i+2}\ \ \ge\ \ k\qquad\text{for all }0 \le i \le s - 1.
    \]
    (In other words:
    passing from $h_{2i}$ to $h_{2i+1}$ and from $h_{2i+1}$ to $h_{2i+2}$
    amounts to multiplications with elements $g^{\alpha}$ and
    $ g^{\beta}$, $0 \leq \alpha,\beta < k$,
    so that $\alpha + \beta \geq k$.)
    Let $h_0, h_1, \dots, h_{2s+1}\in \Z_k$. We call
    the element $h_0 \tens \cdots \tens h_{2s+1}$ of~$S_{2s+1}$ 
    \emph{strongly alternating} if there is an $a\in\Z_k$ such that
    $a\tens h_0\tens\cdots\tens h_{2s+1}$ is strongly alternating, i.e.~if
    $h_1\tens\cdots\tens h_{2s+1}$ is strongly alternating and $h_0\ne h_1$.
\end{definition}

\begin{definition}[Alternating elements]
The element $h_0\tens\cdots\tens h_r$ of~$S_r$
is \emph{alternating} if $h_{i+1}\ne h_i$ for all $0\le i<r$.
\end{definition}

\begin{remark} In general, strongly alternating elements are
alternating. The two
notions coincide if and only if $k=2$. In this case we get back
the alternating label patterns introduced in Definition \ref{def:admissible}.
\end{remark}

The strongly alternating elements are $\Z_k$-invariant in the sense that
$x = h_0 \tens h_1 \tens \cdots \tens h_{2s}$ is strongly
alternating if and only if $gx$~is.

After these preparations, we can write down the chain map $f_\bullet\colon S_\bullet\to M_\bullet$
corresponding to the above choice of $K_\bullet \subset M_\bullet$.

The $\Lambda$-linear maps $f_r: S_r \to \Lambda$ are given by
\begin{eqnarray*}
    f_{2s}([h_1|\cdots|h_{2s}]) & \deq &
    \begin{cases}
        e, & \text{if $[h_1|\cdots|h_{2s}]$ is strongly alternating, and}\\
        0, & \text{otherwise,}
    \end{cases}\\[3mm]
    f_{2s+1}([h_1|\cdots|h_{2s+1}]) & \deq & \sigma_i \,
    f_{2s}([h_2|\cdots|h_{2s+1}])\qquad \text{ for $h_1 = g^i$, $0 \leq i < k$.}
\end{eqnarray*}

\begin{proposition}\label{prop:f-chain-map}
    The collection of the maps $f_r$ is a chain map from the standard
    resolution to the minimal resolution, i.e.\ for all~$s\ge0$ the diagrams
    \[
    \begin{CD}
        S_{2s+1} @>\dd>> S_{2s}     \\
        @VVf_{2s+1}V @VVf_{2s}V\\
        M_{2s+1} @>\dd>> M_{2s}     \\
    \end{CD}
    \qquad\text{and}\qquad
    \begin{CD}
        S_{2s+2} @>\dd>> S_{2s+1}     \\
        @VVf_{2s+2}V @VVf_{2s+1}V\\
        M_{2s+2} @>\dd>> M_{2s+1}     \\
    \end{CD}
    \]
    commute:
   \begin{align*}
       f_{2s}(\dd c) &\ \ =\ \ \tau f_{2s+1}(c)
           \qquad\text{ for $c \in S_{2s+1}$ and }\\
           f_{2s+1}(\dd c) &\ \ =\ \ \sigma f_{2s+2}(c)
           \qquad\text{ for $c \in S_{2s+2}$.}
   \end{align*}
\end{proposition}

\begin{proof}
We proceed by induction on~$s$.  
Let $c=[g^r|h_2|\cdots|h_{2s+1}]$,
$0\le r<k$.  If $s=0$ then $f_0(\dd
c)=f_0(\dd[g^r])=f_0(g^r[\,]-[\,])=g^r-e=\tau_r=\tau\sigma_r=\tau
f_1([g^r])=\tau f_1(c)$.  If $s>0$ then by induction $\sigma
f_{2s}(\dd c)=f_{2s-1}(\dd\dd c)=0$, so $f_{2s}(\dd c)\in\ker
m_\sigma=\im m_\tau$, and to prove $f_{2s}(\dd c)=\tau f_{2s+1}(c)$
it suffices to show that for $1\le i\le k-1$ the coefficient of
$\tau_i$ in $f_{2s}(\dd c)$ with respect to the basis~$T$ equals the
coefficient of $\sigma_i$ in $f_{2s+1}(c)$ with respect to the 
basis~$\Sigma$.  Now $f_{2s}(\dd [g^r|h_2|\cdots|h_{2s+1}])$ equals $g^r
f_{2s}([h_2|\cdots|h_{2s+1}])$ plus a multiple of~$e$, so the
coefficient of $g^i$ is $1$ if $[h_2|\cdots|h_{2s+1}]$ is strongly alternating and
$i=r$, and it is $0$ otherwise.  Comparison with the definition of
$f_{2s+1}$ proves the first equation.

Let $c=[g^t|g^r|h_3|\cdots|h_{2s+2}]$, $0\le t,r<k$.  From the first
equation we know that $\tau f_{2s+1}(\dd c)=f_{2s}(\dd\dd c)=0$, so
$f_{2s+1}(\dd c)\in\ker m_\tau=\im m_\sigma$, and to prove
$f_{2s+1}(\dd c)=\sigma f_{2s+2}(c)$ it suffices to show that the
coefficient of $\sigma_{k}$ in $f_{2s+1}(\dd c)$ with respect to the
basis~$\Sigma$ equals the coefficient of $e$ in $f_{2s+2}(c)$ with
respect to the basis~$T$.  Now $f_{2s+1}(\dd c)$ equals $g^t
f_{2s+1}([g^r|h_3|\cdots|h_{2s+2}])$ plus a linear combination of
the~$\sigma_i$ with $1\le i<k$, so the
coefficient of $\sigma_k$, which equals the coefficient of $g^{k-1}$
with respect to the basis $\{e,g,\ldots, g^{k-1} \}$, equals~$1$ if
$t+r\ge k$ and $[h_3|\cdots|h_{2s+2}]$ is strongly alternating and $0$~otherwise.
So it equals~$1$ if
$[g^t|g^r|h_3|\cdots|h_{2s+2}]$ is strongly alternating and $0$~otherwise.  This proves
the second equation.
\end{proof}

\begin{remark} \label{rem:normalized}
The maps $f_r$ are zero on all non-alternating, or \emph{degenerate},
basis elements.  These generate a subcomplex
of the standard resolution and $f_\bullet$
factors through the quotient by this subcomplex. This quotient
is the so called \emph{normalized standard resolution}. The
induced map from the normalized standard resolution to
the minimal resolution is an isomorphism if and only if $k=2$.
In this case we recover exactly the chain map described
in Section \ref{sec:basics}.
 \end{remark}

\section{Labellings and the combinatorial \boldmath$\Z_k$-Stokes theorem}

\label{sec:labelled_complexes}
\label{sec:ZpStokes}\label{sec:stokes}

Fix an integer $k \ge 2$ and consider an (ordered) simplicial complex $X$ with
vertices labelled with elements of $\Z_k \times \N$. This labelling
is a map
\[
    \ell : V \to \Z_k\times\N
\]
defined on the vertex set $V = V(X)$. For a vertex $v \in V$ and $\ell(v) = (s,c) \in
\Z_k \times \N$ we will call $c$ the \emph{color} and $s$ the \emph{sign} of
$v$. A labelling is called \emph{admissible} if the two vertices of an edge
always carry different colors or the same sign (compare
Definition~\ref{def:admissible}).

Let $X$ be a simplicial complex with an admissible $\Z_k \times
\N$-labelling~$\ell$ and let $C_\bullet(X) = C_\bullet(X;R)$ denote its
simplicial chain complex with coefficients in $R$. We define maps
\[
    h^\ell_r\colon C_r(X)  \to S_r
\]
by
\[
    \simp{v_0,\dots,v_r}  \mapsto
    \begin{cases}
        \sign \pi \cdot s_{\pi(0)}\tens\cdots\tens s_{\pi(r)},&
        \text{ for } \pi\in {\rm Sym}(k) \text{ with } c_{\pi(0)}<\cdots<c_{\pi(r)},
        \text{ and}\\
            0,& \text{ if } |\{ c_i:0\le i\le r \}| < r + 1,
    \end{cases}
\]
where $\ell(v_i)=(s_i,c_i)$ for all $i = 0, \dots, r$.

We call $s_{\pi(0)}\tens\cdots\tens s_{\pi(r)}$ the \emph{pattern}
assigned to $\simp{v_0,\dots,v_r}$ by $\ell$. The coefficient $\sign \pi$ amounts to counting
patterns ``according to orientation''.

The family of maps $(h^\ell_r)$ can alternatively be described as the
composition of the chain map
\[
    \ell_\# : C_\bullet(X) \to C_\bullet((\Z_k)^{*\N})
\]
induced by the map $X \to (\Z_k)^{*\N}$ determined by the labelling $\ell$ and the map of chain complexes
\[
    h_\bullet : C_\bullet( (\Z_k)^{*\N} ) \to S_\bullet
\]
which is given on the ordered simplices by
\[
\simp{ (s_0,c_0), (s_1,c_1), \dots, (s_r,c_r) } \mapsto s_0 \tens s_1
\tens \cdots \tens s_r,
\]
with $c_0 < c_1 < \cdots < c_r$. Hence, the map $h^{\ell}_\bullet$ is itself a map of chain complexes.

Recall the chain map
\[
      f_\bullet : S_\bullet \to M_\bullet
\]
from Section \ref{sec:chain_map}. The combinatorial Stokes theorem is now
a consequence of the fact that the chain map
\[
     f_\bullet \circ h^{\ell}_{\bullet} : C_\bullet(X)  \to M_\bullet
\]
commutes with differentials:  For $x \in C_{r}(X)$, $r \geq 1$, we have
\begin{eqnarray*}
     f_{r-1}(h^{\ell}_{r-1} (\partial x)) & = & \sigma f_r (h^{\ell}_r (x))
     \quad\textrm{for }r\textrm{ even,} \\
     f_{r-1}(h^{\ell}_{r-1} (\partial x)) & = & \tau   f_r (h^{\ell}_r (x)) 
     \quad\textrm{for }r\textrm{ odd.}
\end{eqnarray*}
In order to obtain a counting formula, we compose the maps occuring in these 
equations with the evaluation at $e \in \Z_k$,  
\[
    u :\ \Lambda\ \to\ R, \qquad 
  \sum_{i=0}^{k-1} \alpha_i \cdot g^i\ \mapsto\ \alpha_0 , 
\]
and --- together with the explicit description of $f_{\bullet}$ --- obtain

\begin{theorem}[Combinatorial Stokes formula]\label{thm:stokes}
Let $X$ be a simplicial complex with an admissible
$\Z_k\times\N$-labelling~$\ell$ and let $x\in C_r(X)$ be an $r$-chain.
Then depending on the parity of $r$, we have the following identities:

\begin{itemize}
    \item ($r=2s$). The number of label patterns $h_0 \otimes \ldots \otimes h_{2s-1}$ in $\partial x$
                    so that $g \otimes h_0 \otimes \ldots \otimes h_{2s-1}$ is strongly alternating
                    equals the sum of all strongly alternating label patterns occuring in~$x$.
    \item ($r = 2s+1$). The number of label patterns $h_0 \otimes \ldots \otimes h_{2s}$ occuring in $\partial x$
                        that are strongly alternating and satisfy $h_0 = e$ is
                        equal to the number of label patterns $h_0 \otimes \ldots \otimes h_{2s+1}$
                        occuring in $x$ so that $e \otimes h_0 \otimes \ldots \otimes h_{2s+1}$ is strongly
                        alternating
                        minus the number of label patterns $h_0 \otimes \ldots \otimes h_{2s+1}$
                        occuring in $x$ so that $g \otimes h_0 \otimes \ldots \otimes h_{2s +1}$ is strongly alternating.
\end{itemize}
Here all label patterns are counted with multiplicities and according to orientation.
\end{theorem}

It is remarkable, and not clear a priori, that our approach via chain
complexes and chain maps leads to a counting formula of the stated
form, where --- apart from possible multiplicities imposed by the chain $x$
itself --- all relevant label patterns are counted with multiplicities
$\pm 1$.

For $k=2$, we recover the classical Fan theorem mentioned in the
introduction after Definition \ref{def:admissible}.  If we replace the
evualuation map $u$ by evaluation at another group element, we 
obtain the above identities with all labels shifted cyclically.

\section{Equivariant labellings and \boldmath$\Z_k$-Tucker lemmas}
\label{sec:Tuckerlemmas}

Even though the group $\Z_k$ has played an important role in the
definiton of the objects of Section~\ref{sec:chain_map}, group actions
did not occur in the results of Section~\ref{sec:stokes}.  We will now
consider a simplicial complex $X$ with $\Z_k$ acting on it as a group
of simplicial homeomorphisms (called a {\em $\Z_k$-complex} for short).
This induces an action of $\Z_k$ on
$C_\bullet(X)$ as a group of chain maps, which makes $C_\bullet(X)$ into a
$\Lambda$-chain complex.

As before, we consider the action of $\Z_k$ on the set of labels $\Z_k \times \N$
by cyclically shifting the signs, i.e. $g (s,c) := (gs,c)$. With this
action we say that a labelling $\ell$ on a $\Z_k$-complex $X$ is
\emph{equivariant} if $\ell(gv) = g \ell(v)$ for all $g \in \Z_k$ and
all vertices~$v$ of~$X$.

An equivariant labelling on $X$ can only exist if $X$ is a free
$\Z_k$-space.

If $X$ a $\Z_k$-complex with an admissible equivariant labelling
$\ell$,  then the chain map  $h_\bullet^\ell$ considered in the last section
is obviously $\Lambda$-linear.

\begin{definition} \label{def:gensphere} Let $X$ be a free $\Z_k$-complex and let $r \geq 0$.
A {\em generalized $r$-sphere}
in $C_\bullet(X)$
is a sequence $(x_i)_{0 \leq i \leq r}$ of chains $x_i \in C_i(X)$
satisfying
\[\dd x_{i}\ \ =\ \ \begin{cases}
\sigma x_{i-1},&\text{if $i$ is even,} \\
\tau x_{i-1},&\text{if $i$ is odd}
\end{cases}
\]
for all $0 < i \leq r$. 
\end{definition}

The terminology is motivated by the following example.

\begin{example}\label{ex:Sk}
Let $k>2$ and $X$ be the triangulation of
$S^{2m+1}=S^1\join\cdots\join S^1$ obtained by
triangulating each of the $m+1$ copies of $S^1$ as a $k$-gon.
We number the copies starting with~$0$ and for each~$i$, $0\le i\le m$, 
choose a vertex~$u^i$ in the $(m-i)$-th copy. Let $\Z_k$ act on $X$ in 
such a way that each of the $1$-spheres is invariant under the action
and $gu^i$ is a neighbor of~$u^i$.  We denote the oriented edge from
$u^i$ to~$gu^i$ by~$w^i$.  We define several chains in $C_\bullet(X)$,
starting with
\begin{align*}
o_0^i&\deq\tau u^i,&o_1^i&\deq\sigma w^i.
\end{align*}
So $o^i_0=\dd w^i$ is an orientation chain of a $0$-sphere in the
$(m-i)$-th copy of~$S^1$, and $o^i_1$ an orientation chain of
this $1$-sphere.  Setting
\begin{align*}
x_{2s}&\ \ \deq\ \  u^s\join o_1^{s-1}\join o_1^{s-2}\join\dots\join o_1^0,\\
x_{2s+1}&\ \ \deq\ \ w^s\join o_1^{s-1}\join o_1^{s-2}\join\dots\join o_1^0,
\end{align*}
each $x_j$ is the orientation chain of a $j$-disk, and
\begin{align*}
\tau x_{2s}&\ \ =\ \ o_0^s\join o_1^{s-1}\join o_1^{s-2}\join\dots\join o_1^0,\\
\sigma x_{2s+1}&\ \ =\ \ o_1^s\join o_1^{s-1}\join\dots\join o_1^0\\
\end{align*}
are orientation chains of spheres.  We obtain
\[
\dd x_{2s+1}\ =\ \tau x_{2s},\qquad
\dd x_{2s+2}\ =\ \sigma x_{2s+1}.
\]
\end{example}

\begin{example} \label{ex:EZk}
Let $k\ge2$, $d\ge0$.  The construction of the preceding example
translates to $(\Z_k)^{*(d+1)}$, since $\Z_k\join\Z_k$ contains the
barycentric subdivision of a $k$-gon with the natural $\Z_k$-action.
We set
\begin{align*}
u^i&\deq\simp{(e,d-2i)},\\
w^i&\deq\simp{(e,d-2i-1),(g,d-2i)}-\simp{(e,d-2i-1),(e,d-2i)}
\end{align*}
and continue as in Example~\ref{ex:Sk} to obtain
chains $x_i\in C_i((\Z_k)^{*(d+1)})$ for $0\le
i\le d$ satisfying the conditions of Definition~\ref{def:gensphere}.
Again, each $x_i$ is the orientation chain of an $i$-disk, while
$\sigma x_i$ is the orientation chain of an $i$-sphere for odd~$i$ and
$\tau x_i$ is the orientation chain of an $i$-sphere for even~$i$.
\end{example}

Now the generalized Tucker lemma has the following form. As before,
the map $u : \Lambda \to R$ is the evaluation at $e \in \Z_k$.

\begin{theorem}[Generalized $\Z_k$-Tucker lemma]\label{thm:count-transfer}
Let $X$ be a $\Z_k$-complex which is equipped with an
equivariant admissible $\Z_k\times\N$-labelling $\ell$. Let $(x_i)_{0\le i\le r}$ be
a generalized $r$-sphere in $C_\bullet(X)$ for some $r\ge0$.
We set
\[
  \alpha_i \ \ :=\ \ u \big(\sigma \cdot (   f_\bullet \circ  h_\bullet^{\ell})  (x_i) \big).
\]
(For even $i$, this just counts the number of strongly alternating label 
patterns in $x_i$.)
Then \begin{itemize}
    \item the number $\alpha_0$ equals the sum of the coefficients of the
          $0$-simplices in~$x_0$ (and hence does not depend on~$\ell$);
    \item we have $\alpha_i\equiv \alpha_0\pmod k$ for all $0\le i\le r$.
\end{itemize}
\end{theorem}

\begin{remark}
For $k=2$ it is convenient to work over~$R=\Z/2$.  In this case
$\sigma=\tau$ and $\alpha_i$ is the parity of the number of
alternating  in~$x_i$, which equals the parity of the number
of $+$alternating simplices in~$\sigma x_i$.

If $X$ is a centrally symmetric triangulation of the $r$-sphere $S^r$
that refines the hyperoctahedral triangulation, we obtain the Tucker
lemma (Proposition~\ref{prop:tucker_lemma}) by choosing $x_i$ to be the
orientation chain of an $i$-dimensional hemisphere.  Then $x_0$ is a
chain consisting
of a single point, hence $\alpha_i=1\in\Z/2$ for all~$i$, and we
obtain that the number of $+$alternating simplices in~$X$ is odd.
\end{remark}

\begin{proof}[Proof of Theorem \ref{thm:count-transfer}]
The first assertion on the value of~$\alpha_0$ is immediate.
We now show that $\alpha_{i+1}\equiv\alpha_i\pmod k$ for all $0\le i<r$.
For  $0 \leq 2s+1 < r$ this assertion follows by composing the equation
\[
     \sigma  (f h^{\ell} (x_{2s+2}))  =
     f h^{\ell}( \partial x_{2s+2}) =
     f h^{\ell}( \sigma x_{2s+1})  = 
     \sigma ( f h^{\ell} (x_{2s+1})) 
\]
with the map $u$. The first of these equation uses the fact that $f_\bullet$ and $h_\bullet^{\ell}$ are chain
maps, the second one the definition of a generalized sphere and the last one the equivariance of 
 $f_\bullet$ and $h_\bullet^{\ell}$.

Now let $0\le 2s<r$.  In order to show
$\alpha_{2s+1}\equiv\alpha_{2s}\pmod k$, it suffices to establish
\[
   \sigma \big( f_{2s+1}(h^\ell(x_{2s+1}))-f_{2s}(h^\ell(x_{2s})) \big) \in k\Lambda
\]
and because $\sigma^2=k\sigma$, this will be be a consequence of
\[
   f_{2s+1}(h^\ell(x_{2s+1}))-f_{2s}(h^\ell(x_{2s}))\in\im m_\sigma=\ker m_\tau \, .
\]
But indeed, 
\[
    \tau f_{2s+1}(h^\ell(x_{2s+1})) =f_{2s}(h^\ell(\dd x_{2s+1})) = f_{2s}(h^\ell(\tau x_{2s}))
    =\tau f_{2s}(h^\ell(x_{2s}))  
\]
finishing the proof of Theorem \ref{thm:count-transfer}. 
\end{proof}

\begin{remark} In order to put Definition  
\ref{def:gensphere} and Theorem \ref{thm:count-transfer} 
into a more general perspective, we 
observe that the chains $x_i$ of a generalized $r$-sphere
define a $\Lambda$-chain map $x\colon M_\bullet^{\le
  r}\to C_\bullet(X)$, where $M_\bullet^{\le r}$ denotes the truncation of
the minimal resolution in degree~$r$.  Theorem~\ref{thm:count-transfer}
follows from the fact that the chain map $f\cmps
h^\ell\cmps x\colon M_\bullet^{\le r}\to M_\bullet$ is determined up to
homotopy by the induced map $R\isom H_0(M^{\le r})\to H_0(M)\isom R$, 
which is multiplication by~$\alpha_0$. In essence, the inductive and more 
explicit procedure presented above is based on 
a systematic study of the connecting homomorphisms in cohomology 
resulting from the exact short exact sequences
\[
  \sigma C_\bullet(X)\xto{{\rm incl_\ast}}C_\bullet(X)\xto{m_\tau}\tau C_\bullet(X)
\]
and 
\[
   \tau C_\bullet(X)\xto{{\rm incl_\ast}}C_\bullet(X)\xto{m_\sigma}\sigma C_\bullet(X) \, . 
\]
\end{remark}

From Theorem~\ref{thm:count-transfer} we can derive the following
invariance property of $\alpha_i$ under a change of labellings.

\begin{corollary}\label{cor:alpha-independent}
Let $X$ be a free $\Z_k$-complex, let $r\ge0$  and $x\in C_r(X)$.
If $r$ is even, assume that $\dd (\tau x)=0$,  if $r$ is odd,
assume that $\dd (\sigma x)=0$.
For an arbitrary admissible $\Z_k\times\N$-labelling~$\ell$, set
\[
 \alpha := u \big( \sigma\cdot (  f_\bullet \circ h_\bullet^{\ell}) (x) \big)  \, .
\]
Then the congruence class of~$\alpha$ modulo~$k$ does not depend on
the choice of the labelling~$\ell$.
\end{corollary}

\begin{proof}
We will see that there exists a generalized sphere $(x_i)_{0 \leq i \leq r}$ 
 with $x_r=x$.  Consequently
$\alpha=\alpha_r\equiv\alpha_0\pmod k$, and $\alpha_0$ does not depend
on~$\ell$.

The chains $x_i$ can be constructed recursively starting with $x_r=x$.
To see this, assume that for a chain $y$ the condition $\dd(\sigma
y)=0$ holds.  Then $\sigma \dd y=0$, and since $C_\bullet(X)$ is a free
$\Lambda$-complex, this implies the existence of~$\bar y$ with $\dd
y=\tau\bar y$.  It further follows that $\dd(\tau\bar y)=\dd(\dd
y)=0$.  Analogously the condition $\dd(\tau y)=0$ implies the existence of
$\bar y$ with $\dd y=\sigma\bar y$ and $\dd(\sigma\bar y)=0$.
\end{proof}

\begin{corollary}\label{cor:ZkTucker0}
  Let $X$ be any $\Z_k$-equivariant
  subdivision of the simplicial complex~$(\Z_k)^{*(d+2)}$.
  There is a subcomplex $Y$ of $X$, homeomorphic to a $(d+1)$-sphere,
  such that for every admissible equivariant $\Z_k\times\N$-labelling~$\ell$
  of $X$, the number of $(d+1)$-simplices of $Y$ to which $\ell$ assigns
  strongly alternating patterns, counted as in the definition of
  $\alpha_{d+1}$ in Theorem~\ref{thm:count-transfer}, is congruent to~$1$ modulo~$k$.
\end{corollary}

\begin{proof}
Let $\mathop{\rm sd}:C_\bullet((\Z_k)^{*(d+1)})\to C_\bullet(X)$ be the
equivariant subdivision chain map and $x_i\in C_i((\Z_k)^{*(d+2)})$
the chains constructed in Example~\ref{ex:EZk}.  The chains
$\mathop{\rm sd}(x_i)$ satisfy the conditions of
Theorem~\ref{thm:count-transfer} with $\alpha_0=1$.
\end{proof}

We formulate a consequence of this as a non-existence result
for certain equivariant maps.

\begin{definition}\label{def:Ealt}
For $d\ge0$ and $m\ge d+1$,
we denote by $(\Z_k)^{*m}_{\alt\le d}$
the subcomplex of the join~$(\Z_k)^{*m}$
whose facets consist of all simplices
$\simp{i_1,\dots,i_m}$ ($i_j\in\Z_k$) with at most
$d$ jumps, that is, such that
$\#\{j\in[m-1]:i_j\neq i_{j+1}\}\le d$.
\end{definition}

The following is also implied by the Tucker--Fan lemma that
Meunier \cite[Thm.~4]{meunier:_z_fan05} obtained for odd~$k$.

\begin{corollary}\label{cor:ZkTucker1}
  Let $k\ge2$, and let $X$ be any $\Z_k$-equivariant
  subdivision of the simplicial complex~$(\Z_k)^{*(d+2)}$,
  then there is no equivariant simplicial $\Z_k$-map
\[
\ell: X\ \ \longrightarrow\ \  (\Z_k)^{*m}_{\alt\le d}.
\]
\end{corollary}

\begin{proof}
Since all strongly alternating patterns are alternating, an equivariant map
$\ell\colon X\to(\Z_k)^{*m}_{\alt\le d}$ would establish an admissible
$\Z_k\times\N$-labelling of~$X$ in which no $(d+1)$-simplex gets a strongly
alternating
pattern, contradicting Corollary~\ref{cor:ZkTucker0}.
\end{proof}

\begin{remark}
The spaces  $(\Z_k)^{*m}_{\alt\le d}$ will be reconsidered in 
Section~\ref{sec:rainbow}.  In Corollary~\ref{cor:rainbow} we prove the
existence of a $\Z_k$-equi\-variant map from $(\Z_k)^{*m}_{\alt\le d}$
to the $d$-dimensional space $(\Z_k)^{*(d+1)}$.  Thus
Corollary~\ref{cor:ZkTucker1} also follows directly from Dold's theorem
\ref{thm:dold} below.
\end{remark}

Instead of constructing the chains in Theorem~\ref{thm:count-transfer}
explicitly as in Example~\ref{ex:EZk},
we can also give a homological condition that ensures
their existence.  We illustrate this by giving a proof of Dold's
theorem.

\begin{proposition}\label{prop:hom-bound}
Let $X$ be a simplicial complex with a free $\Z_k$-action,
and $R$ be a commutative ring with~$1$ such that $kR\ne R$. Let  $r\ge0$.
If $\widetilde H_i(X;R)\isom0$ for all $i\leq r$ then for every equivariant
admissible $\Z_k\times\N$-labelling there is an $(r+1)$-simplex of $X$
which is labelled with $r+2$ distinct colors and a strongly alternating pattern.
\end{proposition}

\begin{proof}
It will suffice to construct  a generalized $(r+1)$-sphere
 $(x_i)_{0 \leq i \leq r+1}$ with 
$\alpha_0=1$, because the conclusion $\alpha_{r+1} \ne 0$ of 
Theorem \ref{thm:count-transfer} shows the
existence of the desired $(r+1)$-simplex.

Since $\widetilde H_{-1}(X)\isom0$, $X$ is nonempty and we can set
$x_0=\simp v$ for a simplex v, so $\alpha_0=1$.  Then $\tau x_0$ is a
reduced $0$-cycle.  Further, because $\widetilde
H_{0}(X)\isom0$, we can choose $x_1$ with $\dd x_1=\tau x_0$.

Now assume that for some $i$ with  $1\le i\leq r$, the $x_j$ for $j\le i$ are
already chosen.  In case of odd $i$, we have $\dd(\sigma x_i)=\sigma\dd
x_i=\sigma\tau x_{i-1}=0$, and since $H_i(X)\isom0$, there is an
$x_{i+1}$ such that $\dd x_{i+1}=x_i$.
In case of even $i$,
we get $\dd(\tau x_i)=\tau\dd x_i=\tau\sigma x_{i-1}=0$, and since
$H_i(X)\isom0$, there is an $x_{i+1}$ such that $\dd x_{i+1}=x_i$.  In
both cases $x_{i+1}$ with the desired property can be found.
\end{proof}

\begin{theorem}[Dold \cite{Dold}]\label{thm:dold}
Let $X$ and $Y$ be a simplicial complexes with free $\Z_k$-actions.
Let $r\ge 0$ and $R$ be a commutative ring with~$1$ such that $kR\ne R$.
If $\widetilde H_i(X;R)\isom 0$ for all $i\leq r$ and $\dim Y\le r$
then there is no equivariant simplicial map from $X$ to~$Y$.
\end{theorem}

\begin{proof}
The complex $Y$ admits an equivariant admissible
$\Z_k\times\N$-labelling.  No simplex of $Y$ is labelled with more
than $r+1$ colors, since no simplex has more then $r+1$ vertices.
An
equivariant map from $X$ to $Y$ would induce a labelling with these
same properties on~$X$, contradicting Proposition~\ref{prop:hom-bound}.
\end{proof}

\begin{remark}
The ``$\Z_p$-Tucker lemma'' from Ziegler \cite[Lemma 5.3]{Z77}
corresponds to a different type of labelling.
Namely, call an $\Z_k\times\N$-labelling for the vertices
of a simplicial complex $X$ a \emph{weakly admissible} labelling
if there are no $k$ vertices of a $(k-1)$-simplex $\sigma^{k-1}$ that under the
labelling get all the same color (second component), but all differerent
signs (first component).

Such a labelling corresponds to a simplicial map to a label space
$(\partial \sigma^{k-1})^{*\N}$,
an infinite join of boundaries of $(k-1)$-simplices.
The action of $\Z_k$ by cyclically permuting the vertices of $\sigma^k$
is free on the boundary $\partial \sigma^{k-1}$ only if $k=p$ is a
prime.

We leave it as a challenge to construct a chain map from the
chain complex of the corresponding simplicial complex
to the minimal resolution for~$\Z_p$, and to derive a
combinatorial/algebraic proof for \cite[Lemma 5.3]{Z77} from this.
\end{remark}


\section{The \boldmath$\Z_k$-target space for rainbow colorings}\label{sec:rainbow}

It was Fan's basic insight from his 1952 paper \cite{fan52}
that one gets meaningful Tucker lemmas  also for labellings
of the vertices of antipodal $d$-spheres with labels
from $\{\pm1,\pm2,\ldots,\pm m\}$ for $m>d+1$.
With subsequent generalizations from $\Z_2$ to $\Z_k$, and
from $d$-spheres to arbitrary pseudomanifolds (Fan \cite{fan67}),
it now appears that the space $(\Z_k)^{*\N}_{\alt \le d}$ introduced in
Definition~\ref{def:Ealt} is a natural target space for
$\Z_k$-Fan theorems.  Here we determine its homotopy type.

\begin{theorem} \label{thm:rainbow}
All the inclusions of $\Z_k$-spaces
\[
(\Z_k)^{*(d+1)} =
(\Z_k)^{*(d+1)}_{\alt\le d} \ \ \subset\ \
(\Z_k)^{*(d+2)}_{\alt\le d} \ \ \subset\ \ \cdots \ \
(\Z_k)^{*m}_{\alt\le d} \ \ \subset\ \ \cdots \ \
(\Z_k)^{*\N}_{\alt \le d}\ =\! \bigcup_{m\ge d+1}(\Z_k)_{\alt \leq d}^{*m}.
\]
are strong deformation retracts.
\end{theorem}

The proof of Theorem~\ref{thm:rainbow} is based on the following
elementary homotopy theory lemma.

\begin{lemma} \label{technical} Let $X$ be a topological space and let $A \subset X$ be a subspace which is
contractible (as a topological space).
Then $X$ is a strong deformation retract of the space $X \cup_A CA$,
the union along $A$ of  $X$ and the cone over $A$.
\end{lemma}

\begin{proof} Because $A$ is contractible, we have a homotopy
equivalence
\[
    X \cup_A CA\ \  \simeq\ \  X \cup_{\{a\}} CA
\]
where $a \in A$ is some point and $CA$ is glued to $X$ along
a constant map $A \to \{a\}$. Furthermore, a pair of
homotopy inverse maps can be chosen in such a way that their restrictions to $X$
are identity maps and that the homotopies of their
compositions to the respective identity
maps are constant on $X$. Because the south tip of the unreduced
suspension $\Sigma A = CA / A$ is a strong deformation retract of~$\Sigma A$
($A$ being  contractible), the result follows.
\end{proof}

\begin{proof}[Proof of Theorem~\ref{thm:rainbow}] We fix $k \geq 2$ and start
with some  general observations.
For simplicity, we write $\Z_k$ as $\set{0,1,\dots,k-1}$
instead of $\set{e,g,\dots,g^{k-1}}$ in this section.

Let  $d \geq 0$ and $m \geq d+1$.  We define  $C_{d,m+1,i} \subset (\Z_k)^{*(m+1)}_{\alt \le d}$ as the
(closed) star of the vertex $i \in \Z_k$, where we identify  $\Z_k$ with
its $(m+1)$st copy in $(\Z_k)^{*(m+1)}_{\alt \le d}$. By the definition
of  joins, we can think of $C_{d,m+1,i}$ as
the cone over $C_{d,m+1,i} \cap (\Z_k)^{*m}_{\alt \le d}$ and furthermore know that the intersections
$C_{d,m+1,i} \cap C_{d,m+1,j}$ are contained in $(\Z_k)^{*m}_{\alt \le d}$ for $i \neq j$.

By induction on $d > 0$, we will now prove that each of the intersections
$C_{d,m+1,i} \cap (\Z_k)^{*m}_{\alt \le d}$ is a contractible space (for all $m \geq d+1$). Together with
Lemma \ref{technical} this implies in particular that the inclusion
\[
    (\Z_k)^{*m}_{\alt \le d}\ \ \hookrightarrow\ \  (\Z_k)^{*(m+1)}_{\alt \le d}
\]
is a strong deformation retract, thus proving the theorem.

For $d=0$ and $m \geq 1$, each intersection $C_{d,m+1,i} \cap (\Z_k)^{*m}$ is a full
$(m-1)$-dimensional simplex  $\simp{i,i,i,\ldots, i}$ and
hence contractible.

Now let $d > 0$ and  $m \geq d+1$. For symmetry reasons it is enough to show that the intersection
\[
    P\ \  :=\ \  C_{d,m+1,0} \cap (\Z_k)^{*m}_{\alt \le d}
\]
is contractible. We can write the polyhedron $P$ as the union of two subpolyhedra $P_0$ and $P_{>0}$ defined
as follows: The facets of~$P_0$ are all the facets of $(\Z_k)^{*m}_{\alt \le d}$ whose $m$th vertex
(with respect to the join construction) is equal to $0$. It can be identified with
the closed star in $(\Z_k)^{*m}_{\alt \le d}$ over this vertex and is therefore contractible.
The facets of~$P_{> 0}$ are all the facets of $(\Z_k)^{*m}_{\alt \le (d-1)}$ whose $m$th
vertex is contained in the set $\{1,2, \ldots, k-1\} \subset
\Z_k$. We will show that
\[
    P_0  \ \ \hookrightarrow \ \   P
\]
is a strong deformation retract. Because $P_0$ is contractible, this finally implies contractibility of~$P$.%

We can write
\[
   P_{> 0}\ \  =\ \  (\Z_k)^{*(m-1)}_{\alt \le (d-1)} \cup (C_{d-1,m,1} \cup C_{d-1,m,2} \cup \ldots \cup C_{d-1,m,k-1})
\]
By our induction hypothesis, for each $1 \leq i \leq k-1$ the
intersection
\[
  C_{d-1,m,i} \cap (\Z_k)^{*(m-1)}_{\alt \le (d-1)}
\]
is contractible. Hence, using Lemma \ref{technical} again,
the inclusion
\[
     (\Z_k)^{*(m-1)}_{\alt \le (d-1)}\ \  \hookrightarrow\ \  P_{>0}
\]
is a strong deformation retract. Because on the other hand
\[
    (\Z_k)^{*(m-1)}_{\alt \le (d-1)}\ \  \subset\ \  P_0 ,
\]
this shows that $P_0 \hookrightarrow P_0 \cup P_{ >0} =P $ is a strong deformation retract.
\end{proof}

\begin{corollary}\label{cor:rainbow} For $m \geq  d+1$
the spaces $(\Z_k)^{*(d+1)}$, $(\Z_k)^{*m}_{\alt\le d}$  and
$(\Z_k)^{*\N}_{\alt \le d}$ are all $\Z_k$-homotopy equivalent, in
particular there exists a $\Z_k$-map
$(\Z_k)^{*\N}_{\alt \le d}\to(\Z_k)^{*(d+1)}$.
\end{corollary}

\begin{proof}
The inclusion map from $(\Z_k)^{*(d+1)}$ into any of the
other  spaces is
$\Z_k$-equivariant and a homotopy equivalence by
Theorem~\ref{thm:rainbow}.  Since all of the spaces are free
$\Z_k$-spaces, a theorem of Bredon \cite[Ch.~II]{bredon-cohom}
\cite[Sec.~II.2]{dieck} implies that
these maps are $\Z_k$-homotopy equivalences (i.e. they have equivariant homotopy
inverses, and the homotopies can also be chosen as equivariant maps).
\end{proof}

\begin{small}

\end{small}

\end{document}